\documentclass[final,5p,times]{elsarticle}
\usepackage{graphicx}
\usepackage{amsmath} % assumes amsmath package installed
\usepackage{amssymb}  % assumes amsmath package installed
\usepackage{amsthm}

\usepackage{color}

\usepackage{pgfplots}
\pgfplotsset{compat=newest}
\usepackage{tikz}
\usepackage{tikzscale}
\usetikzlibrary{external}
\usetikzlibrary{pgfplots.groupplots}
\usepackage{array}
\usepackage{booktabs}
\usepackage{multirow}
\usepackage{csquotes}
\usepackage{dblfloatfix}
\usepackage{float}
\usepackage{nicefrac}

\newcommand{\vek}[1]{{\mathbf #1}}

\newcommand{\sv}[1]{\boldsymbol{#1}}

\newcommand{\play}[1]{^{(#1)}}
\newcommand{\Tplay}[1]{^{(#1)T}}

\newcommand{\eg}{e.g.~}

\newcommand{\normTwo}[1]{\left\lVert#1\right\rVert_2}
\newcommand{\APlay}{\mathcal{P}}
\newcommand{\optPlayVec}[2]{{\sv{#1}\play{#2}}^*}

\newtheorem{deff}{Definition}
\newtheorem{thm}{Theorem}

\newcommand{\ff}[1]{\textcolor{black}{#1}}

\begin{document}

	\begin{frontmatter}
	
	
	
	\title{On the Upper Bound of Near Potential Differential Games}

	\author{Balint Varga} 
	
	\address{Institute of Control Systems, Karlsruhe Institute of Technology, Kaiserstrasse 12, 76131 Karlsruhe, Germany, balint.varga2@kit.edu}

	\begin{abstract}
		This letter presents an \ff{extended} analysis and a \ff{novel upper bound }of the subclass of Linear Quadratic Near Potential Differential Games (LQ NPDG). LQ NPDGs are a subclass of potential differential games, for which a distance between an LQ exact potential differential game and the LQ NPDG. 
		\ff{LQ NPDGs exhibit a unique characteristic: the smaller the distance from an LQ exact potential differential game, the more closer their dynamic trajectories.}
		\ff{This letter introduces a novel upper bound for this distance.  Moreover, a linear relation between this distance and the resulting trajectory errors is established, opening the possibility for further application of LQ NPDGs.}
		
	\end{abstract}

	\begin{keyword}	Differential Games; Potential Games; Near Potential Differential Games; Upper Bound	\end{keyword}

%	\begin{highlights}
%		\item Novel, linear upper bound for Linear Quadratic Near Potential Differential Games
%		\item Linear relation between this distance and the resulting trajectory errors for Linear Quadratic Near Potential Differential Games  
%		\item Discussion on the practical applications of Near Potential Differential Games
%	\end{highlights}
	
\end{frontmatter}

\section{Introduction}
Game theory is a widely used mathematical tool to model interaction between multiple agents \cite{2018_HandbookDynamicGame_basar}. \ff{In a \textit{game}, different \textit{players} interact with each other in order to optimize their own \textit{cost function}. Due to the interaction between them, the optimal solution has to be computed in a coupled manner. One of the solution concepts is the so-called \textit{Nash Equilibrium} (NE), which emerges as a solution in non-cooperative games where players independently pursue their goals without forming agreements} \cite[Chapter~7-8]{2005_LQDynamicOptimization_engwerda}. This necessitates coupled optimization processes for each player in an $N$-player game. For a comprehensive overview of the theory of dynamic games, it is referred~to~\cite{2022_TheoryApplicationsDynamic_parilina}.

In \ff{the} case of the so-called \textit{potential games}, the game can be characterized by one single cost (potential) function instead of $N$, coupled optimizations. 
\ff{This enables the calculation of the Nash Equilibrium (NE) by simply optimizing this potential function. Furthermore, the uniqueness of the NE is assured when dealing with a convex potential function, enhancing the appeal of using this game characterization in practical scenarios, like motion planning }\cite{2020_MultiVehicleAutomatedDriving_fabiani}, communication network management \cite{2016_DynamicPotentialGames_zazo}, modeling human-robot interactions \cite{2021_OrdinalPotentialDifferential_varga}, \ff{multi agent systems }\cite{chen2023nash} or \ff{network-flow control problems} \cite{2024_StructureFeedbackPotential_prasad}. %2023_OpenLoopFeedbackLQ_scarpa

\ff{The core idea of near potential games is the usage of a distance metric between two differential games. In that way, the required exactness of the exact potential differential games is transformed into a less restrictive condition, which permits a small, remaining difference between the two games. The concept of near potential static games is introduced in~\cite{2010_ProjectionFrameworkNearpotential_candogan, 2013_NearPotentialGamesGeometry_candogan}. Based on the intuitive idea that if two games are \textit{close} in terms of the properties of the players' strategy sets, their properties in terms of NE should be somehow similar. A systematic framework for static games was developed in~\cite{2010_ProjectionFrameworkNearpotential_candogan}. It was shown that a near potential static game has a \textit{similar} convergence of the strategies\footnote{Note that the convergence of static games means the convergence of the decision-making process, which leads to one of the NEs of the game. The term \textit{dynamics} has no relation to the dynamics of the system states in the context of differential games.} compared to an exact potential static game. A similar convergence of the strategies means that similar changes in the input strategies lead to similar changes in the payoffs in the game. Furthermore, it is also shown that the meaning of close can be quantified in the developed framework, see \cite{2010_ProjectionFrameworkNearpotential_candogan}.} 

%In contrast to static games, NPDGs are suitable to model games with an underlying dynamic system, called \textit{differential games}, which are commonly used to model engineering applications, see e.g.~\cite{2022_ExperimentalEvaluationGameTheoretic_na},
\ff{In this letter, a specific subclass, the \textit{Near Potential Differential Games} (NPDG) is discussed. In \cite{2023_LimitedInformationShared_varga}, the concept of the NPDGs was introduced, in which, the \textit{similarities} of the trajectories are given as a non-linear function of the \textit{closeness} of two games. In this letter, a novel upper bound is provided: A linear relation is derived facilitating a more feasible application of this upper bound. The primary contribution is the derivation of this novel upper bound for NPDGs.}

%\rr{WORIDNG ACCOCIATED LQR PROBLEM!}
\section{Preliminaries}
In the following, the focus of this letter lies on the linear quadratic (LQ) differential games. LQ differential games are useful for modeling a wide range of engineering problems since they provide a simple and effective way to trade off conflicting objectives and make optimal decisions across dynamic systems.
\subsection{Exact Potential Differential Games}
\begin{deff}[LQ Differential Game \cite{2016_NonzeroSumDifferentialGames_basar}] \label{def:diff_game}
		An \textbf{LQ Differential Game} $\Gamma_\mathrm{d}$ is defined as a tuple of 
		\begin{itemize}
		\item a set of $N$ players $\ff{i \in \mathcal{P}=\{1,2,...,N\}}$, 
		\item a dynamic system with the system matrix $\vek{A}$ and the input matrix of player $i$, $\vek{B}\play{i}$
		\begin{equation} \label{eq:sys_dyn_lq}
			\dot{\sv{x}}(t) = \vek{A}\sv{x}(t) + \sum_{i \in \mathcal{P} } \vek{B}\play{i}\sv{u}\play{i}(t),
		\end{equation} 
		%where $\left(\vek{A}, \sum_{i \in \mathcal{P} } \vek{B}\play{i}\right)$ is 
		\item the joint set of control strategies of the players $\,\mathcal{U}  =  \mathcal{U}\play{1} \times ... \times \mathcal{U}\play{N}$ and 
		\item the set of the players' cost functions $\mathcal{J}=  \{J\play{1}, \, ... \,,J\play{N}$\}, where
		\begin{align} \label{eq:cost_fucntion} \nonumber
	J\play{i}&=\frac{1}{2}\int_{0}^{\tau_\mathrm{end}}  \sv{x}(t)^\mathsf{T}  \vek{Q}\play{i}\sv{x}(t) \\
		 &+   \sum_{j \in \APlay} {\sv{u}\play{j}(t)}^\mathsf{T} \vek{R}\play{ij}\sv{u}\play{j}(t)  \text{ d}t, \; i \in \APlay,
		\end{align}
	where $\vek{Q}\play{i}$ and $\vek{R}\play{ij}$ represent the penalty matrices for the system states and system inputs of the player $i$. The end of the game is $\tau_{\mathrm{end}}$. It is assumed that the matrices of the cost functions have a diagonal structure
	$\vek{Q}\play{i} = \mathrm{diag}\left(q\play{i}_1,q\play{i}_2,...,q\play{i}_n\right)$ and 
	$\vek{R}\play{ij} = \text{diag}\left(r\play{ij}_1,r\play{ij}_2,...,r\play{ij}_{p_i}\right),$ are positive semi-definite and positive definite, respectively.
	\end{itemize}
\end{deff}

\begin{deff}[Nash Equilibrium \cite{2016_NonzeroSumDifferentialGames_basar}]\label{def:NE}
		The game is in a \textit{Nash equilibrium} (NE) if the players cannot deviate from their actual strategies without increasing their costs
	$$
	J\play{i}\left( \optPlayVec{u}{i}, \optPlayVec{u}{\neg i}\right) \leq J\play{i}\left( \sv{u}\play{i}, \optPlayVec{u}{\neg i}\right) \; \; \forall i \in \mathcal{P}.
	$$
\end{deff}
In order to compute the NE of a differential game, the so-called coupled Riccati equations are set up \cite[Chapter~7]{2005_LQDynamicOptimization_engwerda}, for which the Hamiltonians of the players are computed such as
\begin{align} \label{eq:ham_SOA} \nonumber
	H\play{i} &= \frac{1}{2}\sv{x}(t)^\mathsf{T}  \vek{Q}\play{i}\sv{x}(t) \\
	&+\frac{1}{2}   \sum\limits_{j \in \mathcal{P}} {\sv{u}\play{j}(t)}^\mathsf{T} \vek{R}\play{ij}\sv{u}\play{j}(t)	+\sv{\lambda}\Tplay{i}(t)\dot{\sv{x}}(t).
\end{align}
For further details on the solution to the coupled Riccati equation, it is referred to \cite[Chapter 3]{2018_HandbookDynamicGame_basar}.

\begin{deff}[LQ Exact Potential Differential Games \ff{\cite{2016_SurveyStaticDynamic_gonzalez-sanchez}}] \label{def:EPG}
		Let an LQ differential game $\Gamma_\mathrm{epd}$ with system dynamics~\eqref{eq:sys_dyn_lq} be given. Furthermore, let the quadratic cost functions~\eqref{eq:cost_fucntion} and Hamiltonian functions \eqref{eq:ham_SOA} of the players be given. Assume that the aggregated inputs of the players and the aggregated input matrices are defined such that
		\begin{align*} 
			\sv{u}\play{p}(t) &= \left[{\sv{u}\play{1}}^\mathsf{T}(t), \, {\sv{u}\play{2}}^\mathsf{T}(t), \,... \, {\sv{u}\play{N}}^\mathsf{T}(t)\right]^\mathsf{T}, \\
			 \vek{B}\play{p} &= \left[\vek{B}\play{1}, \vek{B}\play{2}, ..., \vek{B}\play{N}\right],
		\end{align*}
		respectively. Furthermore, consider an LQ optimal control problem over an infinite time horizon $\tau_\mathrm{end} \rightarrow \infty$ with the cost function
		\begin{equation} \label{eq:pot_cost_function}
			J\play{p}=\frac{1}{2}\int_{0}^{\tau_\mathrm{end}} \sv{x}^\mathsf{T}(t) \vek{Q}\play{p}\sv{x}(t) + {\sv{u}\play{p}}^\mathsf{T}(t)\vek{R}\play{p}\sv{u}\play{p}(t) \mathrm{ d}t
		\end{equation} 
		as well as the Hamilton function 
		\begin{equation} \label{eq:potential_hamil_quad_def}
			H\play{p}(t) = \frac{1}{2}{\sv{x}(t)}^\mathsf{T} \vek{Q}\play{p}\sv{x}(t) +  \frac{1}{2}{\sv{u}\play{p}}^\mathsf{T}(t)\vek{R}\play{p}\sv{u}\play{p}(t) + \sv{\lambda}\Tplay{p}\dot{\sv{x}}(t),
		\end{equation}
		where the matrices $\vek{Q}\play{p}$ and $\vek{R}\play{p}$ are positive semi-definite and positive definite, respectively. If 
		\begin{equation} \label{eq:def_e_pot}
			\frac{\partial H\play{p}(t)}{\partial \sv{u}\play{i}(t)} = \frac{\partial H\play{i}(t)}{\partial \sv{u}\play{i}(t)}
		\end{equation} 
		holds for $\forall i \in \mathcal{P}$, the LQ differential game $\Gamma_\mathrm{epd}$ is an LQ \textbf{exact potential differential game}, which has the potential function~$J\play{p}$. 

\end{deff}

Definition \ref{def:EPG} reveals that the NE can be computed by the optimal control problem of \eqref{eq:sys_dyn_lq} and \eqref{eq:pot_cost_function} in the case of an exact potential differential game as long \eqref{eq:def_e_pot} holds. For further discussions and examples, the reader is referred to \cite{2016_PotentialGameTheory_la} and \cite{2018_PotentialDifferentialGames_fonseca-morales}.

%\begin{deff}[Associated Control Problem]\label{def:Ass_LQR}%
%	\eqref{eq:pot_cost_function}
%	\rr{DEfined this as associated control problem}
%	\rr{Use the de to explain how the trajectories are compuated!!!}
%\end{deff}

%- \rr{That can make it easier to understand, what is what!!!}

\subsection{Distance between two \ff{Potential }Differential Games} %\cf Definition \ref{def:EPG}
Similar to the static case \cite{2013_NearPotentialGamesGeometry_candogan}, a distance measure between two differential games is introduced. 
\begin{deff}[Differential Distance \cite{2023_LimitedInformationShared_varga}]\label{def:Diff_Dist}
		Let an exact potential differential game $\Gamma_\mathrm{epd}$ with the potential function~$J\play{p}$ be given. Furthermore, let an arbitrary LQ differential game $\Gamma_\mathrm{npd}$ according to Definition \ref{def:diff_game} be given. The \textit{differential distance} (DD) between $\Gamma\play{p}_ \mathrm{epd}$ and $\Gamma_\mathrm{npd}$ is defined as
		\begin{equation} \label{eq:diff_distance}
			\sigma\play{i}_d(t):= \left\lVert\frac{\partial H\play{p}(t)}{\partial \sv{u}\play{i}(t)} - \frac{\partial H\play{i}(t)}{\partial \sv{u}\play{i}(t)}\right\rVert_2, \; i \in \mathcal{P}.
	\end{equation}
\end{deff} 
\textbf{Note \ref{def:Diff_Dist}:} Definition \ref{def:Diff_Dist} defines vector space, in which two games can be compared and their "closeness" can be quantified. 
It is the intuitive extension of Definition \ref{def:EPG} because for an exact potential differential game,
$$
\sigma\play{i}_d(t) = 0, \forall t \in [0, \tau_{\mathrm{end}}]
$$
holds, meaning that $\Gamma_\mathrm{npd}$ has the same characteristics as $\Gamma_{\text{ed}}$. Softening the condition $\sigma\play{i}_d(t) = 0$ enables a broader use. Using Definition \ref{def:Diff_Dist}, the subclass of NPDGs is formally defined.

\begin{deff}[Near Potential Differential Game \ff{\cite{2023_LimitedInformationShared_varga}}] \label{def:NPDG_distance}
		A differential game $\Gamma_\mathrm{npd}$ is said to be an NPDG if the DD between $\Gamma_\mathrm{npd}$ and an arbitrary exact potential differential game $\Gamma_\mathrm{epd}$ is 
		\begin{align} \label{eq:def:NPDG_distance}
			\underset{i}{\mathrm{max }} \; \sigma\play{i}_d(t)  \leq \Delta, \; \; i \in \mathcal{P},
		\end{align}
		\ff{where $\Delta \geq 0$ is a small constant, meaning that
		$$
		\underset{\Delta \rightarrow 0}{\mathrm{lim}} \underset{i}{\mathrm{max }} \;\sigma\play{i}_d(t) = 0
		$$
		 holds.} 

\end{deff}

\textbf{Note \ref{def:NPDG_distance}.1:} Definition \ref{def:NPDG_distance} does not exclude the subclass of exact potential differential games as $\Delta = 0$ is possible. Thus, exact potential differential games are a subset of NPDGs.

\textbf{Note \ref{def:NPDG_distance}.2:} The maximum DD is the measure of the likeness between the games. As the maximum DD increases, the dynamics of states and input trajectories of the NPDG are gradually getting larger. Thus, the main question is that for a given upper bound $\Delta$, how large the perturbation of the state and inputs dynamics between $\Gamma_\mathrm{npd}$ and $\Gamma_\mathrm{epd}$ \ff{is admissible}. Therefore, this perturbation is quantitatively characterized for LQ differential games in the following. 

\section{Upper Bound of NPDGs}
The main results of this letter \ff{are} presented in this section: The novel upper bound of the DD and a further analysis of the boundness of an NPDG. 

\subsection{Properties of an NPDG}
\begin{thm}
	[LQ NPDG]
	Let an LQ exact potential differential game $\Gamma\play{p}_{\text{ed}}$ with its state trajectories $\sv{x}\play{p}(t)$ in its NE be given.
	\ff{Furthermore, let an arbitrary LQ differential game $\Gamma_\mathrm{npd}$ according to Definition \ref{def:diff_game} with its state trajectories~$\sv{x}^*(t)$ in the NE of $\Gamma_\mathrm{npd}$ be given. It is also assumed that there is a $\Delta \sv{x}\play{p}(t)\geq 0 $ such
	\begin{equation} \label{eq:chap4_Dxp_NPDG1}
		\sv{x}\play{\mathrm{p}}(t) = \sv{x}^*(t) + \Delta \sv{x}\play{p}(t) \;  \mathrm{or} 
	\end{equation}
	\begin{equation} \label{eq:chap4_Dxp_NPDG2}
		\sv{x}\play{\mathrm{p}}(t) = \sv{x}^*(t) - \Delta \sv{x}\play{p}(t) \; \phantom{or}
	\end{equation}
	hold $\forall t \in [0,\tau_{\mathrm{end}}]$.} If
	\begin{align} \label{eq:chap4_Lemma_1_NPDG}
		\underset{i}{\mathrm{max }}\left\lVert{\vek{B}\play{i}}^\mathsf{T}\vek{P}\play{p} - {\vek{B}\play{i}}^\mathsf{T}\vek{P}\play{i} \right\rVert_2 < \ff{\Delta^*(\Delta)}
	\end{align}
	holds, where $\Delta$ is defined in \eqref{def:NPDG_distance}. 
	\ff{Furthermore, $\vek{P}\play{p}$ is the Riccati matrix obtained from the optimum of the potential function~\eqref{eq:pot_cost_function}.
	The matrix $\vek{P}\play{i}$ is the solution of the coupled Riccati equation~\eqref{eq:ham_SOA} for the player $i$, see \cite{2016_NonzeroSumDifferentialGames_basar}.
	Then $\Gamma_\mathrm{npd}$ is an LQ NPDG in accordance with Definition \ref{def:NPDG_distance}.}
	
\end{thm}

\begin{proof}
	%For the proof, the two partial derivatives of the Hamiltonians in (\ref{eq:diff_distance}) are analyzed using the definition of the LQ NPDGs.	
	The derivative of $H\play{i}$ is expressed as 
	\begin{equation} \label{eq:deriv_player_i}
		\frac{\partial H\play{i}(t)}{\partial \sv{u}\play{i}(t)} = \vek{R}\play{ii}\sv{u}\play{i}(t) + \vek{B}\Tplay{i}\sv{\lambda}\play{i}(t),
	\end{equation}
	which holds for $i \in \mathcal{P}$.
	Since the optimal control law of the players, \eqref{eq:deriv_player_i} is zero, a small perturbation around the optimal solution is sought. Based on \cite{2021_OrdinalPotentialDifferential_varga}, the derivatives of the Hamiltonian of player $i$ can be rewritten as
	\begin{equation} \label{eq:pr_2_res_2_npdg}
		\frac{\partial H\play{i}(t)}{\partial \sv{u}\play{i}(t)} = -\varepsilon\play{i}_c(\sv{x})\vek{B}\Tplay{i}\vek{P}\play{i}\sv{x}^*(t),
	\end{equation}
	and for the derivatives of the Hamiltonian of the potential function
	\begin{equation} \label{eq:diff_Hp}
		\frac{\partial H\play{p}(t)}{\partial \sv{u}\play{i}(t)} = -\varepsilon\play{p}_c(\sv{x}){\vek{B}\play{i}}^\mathsf{T}\vek{P}\play{p}\sv{x}\play{p}(t)
	\end{equation}
	are obtained, where 	
	%$\varepsilon\play{p}_c(\sv{x}) <\!\!<1$ and $\varepsilon\play{i}_c(\sv{x})<\!\!<1$ 
	$$\varepsilon\play{p}_c(\sv{x}) <\!\!<1 \; \mathrm{and} \; \varepsilon\play{i}_c(\sv{x})<\!\!<1$$ 
	are scalar perturbation functions. Substituting the derivatives into \eqref{eq:diff_distance}, the DD is stated as
	$$
	\sigma\play{i}_d(t)=  \left\lVert\varepsilon\play{p}_c(\sv{x}){\vek{B}\play{i}}^\mathsf{T}\vek{P}\play{p}\sv{x}\play{p}(t) - \varepsilon\play{i}_c(\sv{x})\vek{B}\Tplay{i}\vek{P}\play{i}\sv{x}^*(t) \right\rVert_2 .
	$$	
	Introducing an upper bound of the variation 
		$${\varepsilon_c:= \mathrm{max}\left(\varepsilon\play{p}_c(\sv{x}), \varepsilon\play{i}_c(\sv{x})\right)},$$ DD is rewritten as
	\begin{align}\label{eq:NPDG_proof_estim1} \nonumber
		\sigma\play{i}_d(t) &=  \left\lVert\varepsilon_c{\vek{B}\play{i}}^\mathsf{T}\vek{P}\play{p}\sv{x}\play{p}(t) - \varepsilon_c\vek{B}\Tplay{i}\vek{P}\play{i}\sv{x}^*(t) \right\rVert_2 \\
		&\leq \left| \varepsilon_c \right| \left\lVert{\vek{B}\play{i}}^\mathsf{T}\vek{P}\play{p}\sv{x}\play{p}(t) - \vek{B}\Tplay{i}\vek{P}\play{i}\sv{x}^*(t) \right\rVert_2.
	\end{align}
	On the one hand, if \eqref{eq:chap4_Dxp_NPDG1} holds, the upper bound of $\sigma\play{i}_d(t)$ is rewritten to
	\begin{align}\nonumber
		\sigma\play{i}_d&(t)= \\ \nonumber
		 = &\left| \varepsilon_c \right| \left\lVert{\vek{B}\play{i}}^\mathsf{T}\vek{P}\play{p}\sv{x}\play{p}(t) - \vek{B}\Tplay{i}\vek{P}\play{i}\sv{x}\play{p}(t) + \vek{B}\Tplay{i}\vek{P}\play{i}\Delta\sv{x}\play{p}(t)\right\rVert_2 \\  \nonumber
		\leq &\left| \varepsilon_c \right| \left\lVert \left( {\vek{B}\play{i}}^\mathsf{T}\vek{P}\play{p} - \vek{B}\Tplay{i}\vek{P}\play{i}\right)\sv{x}\play{p}(t) \right\rVert_2 \\ \nonumber
		+ &\underbrace{\left| \varepsilon_c \right| \left\lVert\vek{B}\Tplay{i}\vek{P}\play{i}\Delta\sv{x}\play{p}(t)\right\rVert_2}_{\approx 0 \; \mathrm{since} \; \varepsilon_c \cdot \Delta\sv{x}\play{p} <\!\!< 1 \, \mathrm{and} \, \varepsilon_c \cdot \Delta\sv{x}\play{p} \rightarrow 0} \\
		\approx &\left| \varepsilon_c \right| \left\lVert{\vek{B}\play{i}}^\mathsf{T}\vek{P}\play{p}- \vek{B}\Tplay{i}\vek{P}\play{i} \right\rVert_2 \left\lVert \sv{x}\play{p}(t)\right\rVert_2 \; i \in \mathcal{P}.
	\end{align}
	On the other hand, if \eqref{eq:chap4_Dxp_NPDG2} holds, the upper bound of $\sigma\play{i}_d(t)$ is
	\begin{equation} \label{eq:NPDG_proof_estim2}
		\sigma\play{i}_d(t) \leq \left\| \varepsilon_c \right| \left\lVert{\vek{B}\play{i}}^\mathsf{T}\vek{P}\play{i}- \vek{B}\Tplay{i}\vek{P}\play{p} \right\rVert_2 \left\lVert \sv{x}^*(t)\right\rVert_2 \; i \in \mathcal{P}.
	\end{equation}
	Introducing the notation for the maximum magnitude of the state vectors
	$$
	x_\mathrm{max} := \mathrm{max }\left(\left\lVert\sv{x}^*(t)\right\rVert_2, \left\lVert \sv{x}\play{p}(t)\right\rVert_2\right),
	$$ 
	the estimations \eqref{eq:NPDG_proof_estim1} and \eqref{eq:NPDG_proof_estim2} can be combined into
	$$
	\sigma\play{i}_d(t) \leq \left| \varepsilon_c \right| \left\lVert{\vek{B}\play{i}}^\mathsf{T}\vek{P}\play{p}- \vek{B}\Tplay{i}\vek{P}\play{i} \right\rVert_2 x_\mathrm{max} \; i \in \mathcal{P}.
	$$
	Introducing $\Delta^* = \frac{\Delta}{\left| \varepsilon_c \right| \cdot x_\mathrm{max}}$ leads to the upper bound of $\sigma\play{i}_d$,
	$$
	\underset{i}{\mathrm{max }}\left\lVert{\vek{B}\play{i}}^\mathsf{T}\vek{P}\play{p} - {\vek{B}\play{i}}^\mathsf{T}\vek{P}\play{i} \right\rVert_2 < \Delta^*
	$$
	proving that $\,\Gamma_{\mathrm{npd}}$ is an NPDG with an upper bound of $\Delta^*$.
	
\end{proof}	
If the upper bound of DD $\sv{\sigma}_d$ between the NPDG and the exact potential differential games is sufficiently \textit{small}, closed-loop characteristics with similar results can be drawn. In the case of differential games system state trajectories are analyzed\footnote{In the static case, the decision procedure to find the NE is the focus of the analysis. For a given distance between two static games, an approximate NE with an $\epsilon$ limit is obtained, which is called the $\epsilon$-NE of the game. For more information on the near potential static game and the concept of $\epsilon$-Nash Equilibrium, it is referred to \cite{2013_NearPotentialGamesGeometry_candogan}.}. The terms \textit{small} and \textit{similar} are described more precisely in the next subsection.

\subsection{Dynamics of LQ NPDGs} 
\ff{The analysis of the so-called (approximate) $\epsilon$-NE can be found in \cite{firoozi2020convex} or \cite{2023_OpenLoopFeedbackLQ_scarpa}. In this letter, the dynamics of the system trajectories are analyzed in order to provide a bound of the differences between two LQ differential games. In contrast to \cite{2023_LimitedInformationShared_varga}, this letter provides a new, linear relation between the DD and the trajectory error.}

Let it be assumed for the LQ differential game $\,\Gamma_\mathrm{npd}$ that the control laws of the players $i \in \mathcal{P}$ are obtained from the solution to the coupled Riccati equations over an infinite time horizon, which leads to the closed-loop system dynamics
\begin{align}\label{eq:closed_loop_linear_system} 
	&\dot{\sv{x}}(t) = \vek{A}^*_c\sv{x}(t), \; \; \sv{x}(t_0) = \sv{x}_0, \\ \nonumber
	&\mathrm{where} \; \vek{A}^*_c = \vek{A} - \sum_{i \in \mathcal{P}} \vek{B}\play{i}{\vek{R}\play{i}}^{-1} {\vek{B}\play{i}}^\mathsf{T} \vek{P}\play{i}
\end{align}
and that the unique solution to \eqref{eq:closed_loop_linear_system} is
\begin{equation} \label{eq:sol_NE}
	\sv{x}^*(t) = e^{\vek{A}^*_c \cdot t}\sv{x}_0.
\end{equation}
For the LQ exact potential differential games $\,\Gamma_\mathrm{epd}$, the control law $\vek{K}\play{p}= {\vek{R}\play{p}}^{-1} {\vek{B}\play{p}}^\mathsf{T} \vek{P}\play{p}$ is obtained from the optimization of the potential function (\ref{eq:pot_cost_function}), which is used to compute the feedback system dynamics
\begin{align}\label{eq:closed_loop_linear_system_pot}
	&\dot{\sv{x}}\play{p}(t) = \vek{A}\play{p}_c\sv{x}\play{p}(t), \; \; \sv{x}\play{p}(t_0) = \sv{x}\play{p}_0, \\ \nonumber
	&\mathrm{where} \; \vek{A}\play{p}_c = \vek{A} - \vek{B}\play{p}{\vek{R}\play{p}}^{-1} {\vek{B}\play{p}}^\mathsf{T} \vek{P}\play{p}.
\end{align}
The solution to (\ref{eq:closed_loop_linear_system_pot}) is
\begin{equation} \label{eq:sol_pot}
	\sv{x}\play{p}(t) = e^{\vek{A}\play{p}_c \cdot t}\sv{x}\play{p}_0.
\end{equation}
From the state trajectories $\sv{x}\play{p}(t)$ and $\sv{x}^*(t)$, an upper bound ($\eta$) of the errors is provided for a given $\Delta$ between two games. 
For this, a notion of the difference between two closed-loop system behaviors is introduced in Definition \ref{def:closed_loop_error}.

\begin{deff}[Closed-Loop System Matrix Error]\label{def:closed_loop_error}
	Consider an LQ exact potential differential game $\Gamma_\mathrm{epd}$ with the system trajectories \eqref{eq:sol_pot}. Furthermore, assume that an arbitrary LQ differential game $\Gamma_\mathrm{npd}$ is an NPDG with the system trajectories \eqref{eq:sol_NE}. Then, the \textbf{closed-loop system matrix error} between $\Gamma_\mathrm{epd}$ and $\Gamma_\mathrm{npd}$ is defined as
		\begin{equation}
			\Delta \vek{K} := \vek{A}^*_c - \vek{A}\play{p}_c.
		\end{equation}
	
\end{deff}

\textbf{Note \ref{def:closed_loop_error}:}
Two differential games are \textit{similar}, if the closed-loop system matrix error is small and consequently, 
the system trajectories of these two games $\sv{x}^*(t)$ and $\sv{x}\play{p}(t)$ are \textit{close} to each other. In this case, $\Gamma_\mathrm{npd}$ is an NPDG. This closeness between an NPDG and an LQ exact potential differential game is quantified in Theorem \ref{lem:npdg_boundedness}.

\begin{thm}[Boundedness of NPDGs] \label{lem:npdg_boundedness}
		Let an LQ NPDG $\,\Gamma_\mathrm{npd}$ and an exact potential differential game $\,\Gamma_\mathrm{epd}$ be given. Let the system state trajectories of the two games $\,\Gamma\play{p}_\mathrm{epd}$ and $\,\Gamma_\mathrm{npd}$ be  $\sv{x}\play{p}(t)$ and $\sv{x}^*(t)$, respectively. Moreover, 
		\begin{equation} \label{eq:proof_init_value}
			{\sv{x}\play{p}(t_0) = \sv{x}^*(t_0) = \sv{x}_0}
		\end{equation}	
		hold for the initial values. Then, the error between the system state trajectories of $\,\Gamma_\mathrm{npd}$ and $\,\Gamma_\mathrm{epd}$ are bounded over an arbitrary time interval $[t_0,t_1]$, such that
		\begin{equation} \label{eq:dif_traj_with_eps}
			\normTwo{\sv{x}\play{p}(t)-\sv{x}^*(t)} \leq \ff{C_\mathrm{NPDG}(t) \cdot \Delta}, \; \; \forall t \in [t_0,t_1],
		\end{equation}
		where $C_\mathrm{NPDG}(t)\geq 0$ is a positive, time-invariant coefficient.

\end{thm}
\begin{proof}
		From the solution to the differential equations (\ref{eq:closed_loop_linear_system}) and (\ref{eq:closed_loop_linear_system_pot}), 
		$$
		\normTwo{\sv{x}^*(t)-\sv{x}\play{p}(t)}= \normTwo{e^{\vek{A}^*_c \cdot t}\sv{x}_0 - e^{\vek{A}\play{p}_c\cdot t}\sv{x}_0}
		$$ 
		is obtained. As (\ref{eq:proof_init_value}) holds, using Definition \ref{def:closed_loop_error} and \cite[Theorem 11.16.7]{2009_MatrixMathematicsTheory_bernstein} leads to 
		\begin{equation} \label{eq:chap4_NPDG_proof_upper_limit}
			\normTwo{\sv{x}^*(t)-\sv{x}\play{p}(t)} \leq \normTwo{\Delta \vek{K}\cdot t}  e^{ \mathrm{max}\, \left(\normTwo{\vek{A}\play{p}_c\cdot t}, \normTwo{\vek{A}^*_c \cdot t}\right) }  \normTwo{\sv{x}_0}.
		\end{equation}
		In the following, an upper bound of $\Delta \vek{K}$ is sought. Let the notation 
		\begin{equation} \label{eq:chap4_bound_npdg_P}
			\vek{P}_{\sum \mathcal{P}} = \begin{bmatrix}
				\vek{P}\play{1}_{\sum \mathcal{P}} \vspace*{1mm}\\
				\vspace*{1mm}
				\vek{P}\play{2}_{\sum \mathcal{P}} \\
				\vdots \\
				\vek{P}\play{i}_{\sum \mathcal{P}} \\
				\vdots \\
				\vek{P}\play{N}_{\sum \mathcal{P}}
			\end{bmatrix} = 
			\begin{bmatrix}
				{\vek{R}\play{1}}^{-1} {\vek{B}\play{1}}^\mathsf{T} \vek{P}\play{1} \\
				{\vek{R}\play{2}}^{-1} {\vek{B}\play{2}}^\mathsf{T} \vek{P}\play{2} \\
				\vdots \\
				{\vek{R}\play{i}}^{-1} {\vek{B}\play{i}}^\mathsf{T} \vek{P}\play{i} \\
				\vdots \\
				{\vek{R}\play{N}}^{-1} {\vek{B}\play{N}}^\mathsf{T} \vek{P}\play{N}
			\end{bmatrix}
		\end{equation}
		be introduced. Substituting (\ref{eq:sol_NE}), (\ref{eq:closed_loop_linear_system_pot}) and \eqref{eq:chap4_bound_npdg_P}  in \eqref{eq:chap4_NPDG_proof_upper_limit}, the upper bound 
		\begin{align} \label{eq:chap4_NPDG_proof} \nonumber
			\normTwo{\Delta \vek{K}} =& 		\normTwo{ \vek{B}\play{p} {\vek{R}\play{p}}^{-1} {\vek{B}\play{p}}^\mathsf{T} \vek{P}\play{p} - \vek{B}\play{p} \sum_{i \in \mathcal{P}} {\vek{R}\play{i}}^{-1} {\vek{B}\play{i}}^\mathsf{T} \vek{P}\play{i}  }\\ \nonumber
			=&\normTwo{\vek{B}\play{p}\left( {\vek{R}\play{p}}^{-1} {\vek{B}\play{p}}^\mathsf{T} \vek{P}\play{p} - \vek{P}_{\sum \mathcal{P}}  \right)} \\
			=&\normTwo{\vek{B}\play{p} {\vek{R}\play{p}}^{-1} \left({{\vek{B}\play{p}}^\mathsf{T} \vek{P}\play{p} - \vek{R}\play{p}}\vek{P}_{\sum \mathcal{P}}\right)} 		
		\end{align} 
		is obtained. In addition, let the matrix 
		\begin{equation} \label{eq:chap4_bound_npdg_R}
			\vek{R}\play{p} = \begin{bmatrix}
				\vek{R}\play{p}_1, \; 
				\vek{R}\play{p}_2 , \;
				\cdots , \;
				\vek{R}\play{p}_i , \;
				\cdots , \;
				\vek{R}\play{p}_N
			\end{bmatrix}^\mathsf{T}
		\end{equation}
		be defined where $\vek{R}\play{p}_i$ is the submatrix for the inputs $\sv{u}\play{i}$ of player $i$, for which 
		\begin{equation*}
			{\vek{R}\play{p}} \vek{P}_{\sum \mathcal{P}} = \sum_{i \in \mathcal{P}} \vek{R}\play{p}_i \vek{P}\play{i}_{\sum \mathcal{P}}
		\end{equation*}
		hold. Thus \eqref{eq:chap4_NPDG_proof} can be reformulated to 
		\begin{align} \nonumber
			\normTwo{\Delta \vek{K}} =&\normTwo{\vek{B}\play{p} {\vek{R}\play{p}}^{-1} \left({\vek{B}\play{p}}^\mathsf{T} \vek{P}\play{p} -  \sum_{i \in \mathcal{P}} \vek{R}\play{p}_i \vek{P}\play{i}_{\sum \mathcal{P}} \right) } \\
			\leq& \normTwo{\vek{B}\play{p}}\normTwo{{\vek{R}\play{p}}^{-1}}\normTwo{{\vek{B}\play{p}}^\mathsf{T} \vek{P}\play{p} - \sum_{i \in \mathcal{P}} \vek{R}\play{p}_i \vek{P}\play{i}_{\sum \mathcal{P}}}.
		\end{align}
		
		Due to the well-known scaling ambiguity, there is a manifold of the potential functions (\ref{eq:pot_cost_function}) that result in an identical feedback gain matrix, thus a scaling factor $\kappa^{p}>0 \in \, \mathbb{R}$ can be chosen such that  ${\tilde{J}\play{p}} = \kappa^{p} \cdot J\play{p}$ and $%\begin{equation} \label{eq:chap4_npdg_Rp_cond}
			\normTwo{\vek{R}\play{p}} > 1
		$ %\end{equation}
		holds.
		Assuming a suitable scaling, \eqref{eq:chap4_NPDG_proof} leads to
		$$
		\normTwo{\Delta \vek{K}} \leq   \normTwo{\vek{B}\play{p}} \normTwo{  {\vek{B}\play{p}}^\mathsf{T} \vek{P}\play{p} - \sum_{i \in \mathcal{P}} \vek{R}\play{p}_i \vek{P}\play{i}_{\sum \mathcal{P}}  }.
		$$ 
		Then, let the following matrix be introduced 	
		\begin{equation} \label{eq:Froben_norm}
			\tilde{\vek{F}} = \begin{bmatrix}
				{\vek{B}\play{1}}^\mathsf{T} \vek{P}\play{p} -  \vek{R}\play{p}_1 \vek{P}\play{1}_{\sum \mathcal{P}} \\
				\vdots \\
				{\vek{B}\play{i}}^\mathsf{T} \vek{P}\play{p} - \vek{R}\play{p}_i \vek{P}\play{i}_{\sum \mathcal{P}} \\
				\vdots \\
				{\vek{B}\play{N}}^\mathsf{T} \vek{P}\play{p} - \vek{R}\play{p}_N \vek{P}\play{N}_{\sum \mathcal{P}}
			\end{bmatrix} = {\vek{B}\play{p}}^\mathsf{T} \vek{P}\play{p} - \sum_{i \in \mathcal{P}} \vek{R}\play{p}_i \vek{P}\play{i}_{\sum \mathcal{P}}.
		\end{equation}
		The so-called Frobenius norm is defined as the entry-wise Euclidean norm of a matrix (see \cite{2021_NormsPrincipalSubmatrices_bunger}), for which
		\begin{equation} \label{eq:Frobnorm2}
			\normTwo{\tilde{\vek{F}}} \leq \left| \left|  \tilde{\vek{F}}\right|\right|_F
		\end{equation}
		holds (see \cite[Chapter 5]{2017_MatrixAnalysis_horn} or \cite[Section 9.8.12]{2009_MatrixMathematicsTheory_bernstein}). Applying the definition of the Frobenius norm to \eqref{eq:Froben_norm},
		\begin{equation} \label{eq:Frobnorm3}
			\left| \left|  \tilde{\vek{F}}\right|\right|_F = N \cdot \underset{i}{\mathrm{max }} \normTwo{{\vek{B}\play{i}}^\mathsf{T} \vek{P}\play{p} - \vek{R}\play{p}_i \vek{P}\play{i}_{\sum \mathcal{P}}}, \, i \in \mathcal{P}
		\end{equation}
		is obtained. Using property \eqref{eq:Frobnorm2} and \eqref{eq:Frobnorm3} leads to an upper bound 
		\begin{align} 
			\normTwo{\Delta \vek{K}}  &\leq   \normTwo{\vek{B}\play{i}} \normTwo{ {\vek{B}\play{p}}^\mathsf{T} \vek{P}\play{p}  -  \sum_{i \in \mathcal{P}} \vek{R}\play{p}_i \vek{P}\play{i}_{\sum \mathcal{P}} } \\
			&\leq\normTwo{\vek{B}\play{p}}   N \cdot \underset{i}{\mathrm{max}} \, \normTwo{{\vek{B}\play{i}}^\mathsf{T} \vek{P}\play{p} \! - \vek{R}\play{p}_i \vek{P}\play{i}_{\sum \mathcal{P}} }.
		\end{align}
		Due to the scaling ambiguity, ${{\tilde{J}\play{i}} = \kappa^{i} \cdot  J\play{i}}, \, \kappa^{i}>0 \in \, \mathbb{R} $ holds and $\kappa^{i}$ and $\kappa^{p}$ can be modified to obtain $\vek{R}\play{i}$ and $\vek{R}\play{p}$, such that
		\begin{equation*} 
			\normTwo{{\vek{B}\play{i}}^\mathsf{T} \vek{P}\play{p} \! - \! \vek{R}\play{p}_i {\vek{R}\play{i}}^{-1} {\vek{B}\play{i}}^\mathsf{T} \vek{P}\play{i} } \leq \normTwo{{{\vek{B}}^{(i)}}^\mathsf{T} \vek{P}\play{p} -  {\vek{B}\play{i}}^\mathsf{T} \vek{P}\play{i}}
		\end{equation*}	
		holds, for which 
		\begin{equation} \label{eq:chap4_NPDG_extra_codition}
			\normTwo{\vek{R}\play{p} {\vek{R}\play{i}}^{-1} {\vek{B}\play{i}}^\mathsf{T} \vek{P}\play{i} } \geq \normTwo{ {\vek{B}\play{i}}^\mathsf{T} \vek{P}\play{i}}
		\end{equation}
		is sufficient (see \cite[Section 9.9.42]{2009_MatrixMathematicsTheory_bernstein}). This leads to
		\begin{align} \label{eq:chap4_NPDG_lemma_preturb_DK_fin}
			\normTwo{\Delta \vek{K}} &\leq  \normTwo{\vek{B}\play{p}}  N \cdot \underset{i}{\mathrm{max }} \normTwo{{{\vek{B}}^{(i)}}^\mathsf{T} \vek{P}\play{p} -  {\vek{B}\play{i}}^\mathsf{T} \vek{P}\play{i}} \\ \nonumber
			&= \normTwo{\vek{B}\play{p}} N \cdot \Delta.
		\end{align}
		The substitution of the upper bound of $\Delta \vek{K}$ in \eqref{eq:chap4_NPDG_proof_upper_limit} by \eqref{eq:chap4_NPDG_lemma_preturb_DK_fin} leads to the coefficient 
		\begin{equation} \label{eq:chap4_NPDG_x_upper_estim}
			\ff{C_\mathrm{NPDG}(t) = \frac{\normTwo{\sv{x}_0}}{\left| \varepsilon_c \right| \cdot x_\mathrm{max}} \cdot \normTwo{\vek{B}\play{p}} N \cdot t \cdot e^{ \mathrm{max}\, \left(\normTwo{\vek{A}\play{p}_c\cdot t}, \normTwo{\vek{A}^*_c \cdot t}\right) }}, 
		\end{equation}
		which results in the following upper bound of the trajectory error
		\begin{align} \label{eq:chap4_ndpg_lq_fin}
	\ff{		\normTwo{\sv{x}\play{p}(t)-\sv{x}^*(t)} \leq C_\mathrm{NPDG}(t) \cdot\Delta.}
		\end{align}

\end{proof}		
	
\textbf{Remark 1:}\\
From \eqref{eq:chap4_ndpg_lq_fin}, it can be seen that the upper bound of the DD governs the maximal admissible error between the trajectories, where the function $C_\mathrm{NPDG}(t)$ depends only on the initial value, the system structure and the time interval $[t_0,t_1]$.

\textbf{Remark 2:}\\
In \eqref{eq:chap4_NPDG_x_upper_estim}, $C_{\mathrm{NPDG}}(t)$ is bounded in the time interval $[t_0,t_1]$. Thus, \ff{Theorem} \ref{lem:npdg_boundedness} holds ${\forall t \in [t_0,t_1]}$ only. However, $\Delta$ can be defined as
$$
\Delta := \begin{cases}
	\Delta_1 & \forall t \in [t_0,t_1] \\
	\Delta_2 & \forall t \in [t_1,t_2] \\
	\vdots \\
	\Delta_N & \forall t \in [t_{N-1},t_N] \\
	\vdots
\end{cases}
$$
In case of asymptotically stable system state trajectories $\sv{x}\play{p}(t)$ and $\sv{x}^*(t)$, a monotonic decreasing series, $\Delta_{N-1} \geq \Delta_{N},$ can be assumed to prevent $C_{\mathrm{NPDG}}(t)$ from an exponential growth for $t \rightarrow \infty$. Consequently, Theorem \ref{lem:npdg_boundedness} also holds for~$t \rightarrow \infty$.

\textbf{Remark 3:}\\
Note that \ff{Theorem} \ref{lem:npdg_boundedness} differs from the upper bound of the distance between solutions of two general initial value problems of differential equations: The upper bound between two general initial value problems is given as a function of the Lipschitz constant and is usually proved with the Gr\"onwall-Bellman inequality, see \eg \cite[Theorem 3.4.]{2002_NonlinearSystems_khalil}. On the other hand, \ff{Theorem} \ref{lem:npdg_boundedness} provides the link between the upper bound \ff{$\normTwo{\sv{x}\play{p}(t)-\sv{x}^*(t)}$ }and the DD of the two games~$\Delta$, which differs from general initial value problems. Thus, \ff{Theorem} \ref{lem:npdg_boundedness} is a special case of Theorem 3.4.~in~\cite{2002_NonlinearSystems_khalil}. 
%Edgar M.E.Wermuth: Two remarks on matrix exponentials, 1989

%Explicit solutions of Riccati equations appearing in differential games - Lucas Jodar, 1991

%On the existence of Nash strategies and solutions to coupled riccati equations in linear-quadratic games G. P. Papavassilopoulos, J. V. Medanic and J. B. Cruz Jr 

\section{Discussion}
\ff{The main result of this letter enables a broader understanding of the concepts of NPDGs, which provide a more compact representation of strategic games. This makes them suitable for engineering applications, as the strictness of exact potential differential games is softened, thereby extending the applicability of the concept of potential games.}

Illustrative engineering examples include human-human or robot-human interactions, for which NPDGs are suitable models.~Such interactions are modeled by differential games in literature \cite{2019_DifferentialGameTheory_li, 2021_TheoreticalExperimentalInvestigation_na} and studies have demonstrated that the resulting motions of human-human or robot-human interactions can be characterized by the NE of this differential game \cite{2009_NashEquilibriaMultiAgent_braun}. % 2022_ExperimentalEvaluationGameTheoretic_na. 
Nevertheless, the assumption of NE can be violated due to the so-called \textit{bounded rationality} of humans in some cases (cf. \cite{2019_ExperimentsSensorimotorGames_chasnov, 10163400}). In cases where these violations of the NE in human-machine interaction scenarios, the proposed upper bound of the DD is a helpful tool to quantify the deviation from the NE. Thus, the concept can be used to analyze and design human-machine interactions.

\addtolength{\textheight}{-9.0cm}

\section{Summary and Outlook} \label{sec:sum_and_out}

\ff{This letter introduces a novel upper bound between an NPDG and an exact potential differential game. Moreover, this letter shows that the resulting trajectory error has a linear relation to the defined upper bound, which enables the prediction of the maximal trajectory error between an NPDG and an exact potential differential game. In the future, the proposed NPDG will be applied to model human-machine interactions.}

%\rr{FIX REFERENCE!}

%Uberlegen!!!

%\rr{Convex analysis for LQG systems with applications to major–minor LQG mean–field game systems}

%\rr{Links in the commments to cite!!!}

%\rr{Cover letter-t ugy irni, hogy oda tenni}

%\bibliography{npdg_varga}	
%\input{manuscript_varga_npdg_fin.bbl}

\end{document}